\documentclass{amsart}
\usepackage{amsfonts,amsmath,amssymb,amsthm,anyfontsize,array,bm,booktabs,caption,enumitem,extarrows,hyperref,longtable,mathalfa,mathdots,mathtools,microtype,slashed,tabularx,tikz,tikz-cd,parskip,comment,wasysym,subcaption}
\usepackage[nameinlink]{cleveref}
\usetikzlibrary{arrows}

\graphicspath{{pics/}}

\usepackage[sc,osf]{mathpazo}
\usepackage[mathscr]{euscript}

\theoremstyle{definition}
\newtheorem{dfn}{Definition}[section]

\theoremstyle{theorem}
\newtheorem{prp}[dfn]{Proposition}
\newtheorem{lmm}[dfn]{Lemma}

\newtheorem{thm}[dfn]{Theorem}

\newtheorem*{qst}{Question}

\newcounter{mainthms}

\theoremstyle{theorem}
\newtheorem{mainthm}[mainthms]{Theorem}

\numberwithin{equation}{subsection}

\newcommand{\conj}[1]{\overline{#1}}
\newcommand{\wilde}[1]{\widetilde{#1}}
\newcommand{\undl}[1]{\underline{#1}}

\newcommand{\vnorm}[1]{\left\lvert #1 \right\rvert}

\newcommand{\xlra}[1]{\xlongrightarrow{#1}}
\newcommand{\xlla}[1]{\xlongleftarrow{#1}}

\newcommand{\vphi}{\varphi}
\newcommand{\veps}{\varepsilon}

\newcommand{\mc}[1]{\mathcal{#1}}
\newcommand{\mf}[1]{\mathfrak{#1}}
\newcommand{\ms}[1]{\mathscr{#1}}
\newcommand{\mb}[1]{\mathbb{#1}}
\newcommand{\cat}[1]{\mathsf{#1}}

\newcommand{\sC}{\ms{C}}

\newcommand{\Spaces}{\ms{S}}

\newcommand{\sCRW}{\ms{CRW}}

\newcommand{\cMF}{\mc{MF}}

\newcommand{\cO}{\mc{O}}

\newcommand{\cZ}{\mc{Z}}

\newcommand{\cRW}{\mathbf{RW}}

\newcommand{\fe}{\mf{e}}
\newcommand{\bh}{\mathbf{h}}

\newcommand{\bMF}{\mathbf{MF}}
\newcommand{\bT}{\mathrm{T}}

\newcommand{\CC}{\mb{C}}
\newcommand{\KK}{\mb{K}}
\newcommand{\LL}{\mb{L}}

\newcommand{\TT}{\mb{T}}
\newcommand{\ZZ}{\mb{Z}}

\newcommand{\Hom}{\mathrm{Hom}}
\newcommand{\End}{\mathrm{End}}

\newcommand{\Fun}{\mathrm{Fun}}
\newcommand{\Map}{\mathrm{Map}}

\newcommand{\CAlg}{\mathrm{CAlg}}

\DeclareMathOperator{\Spec}{Spec}
\newcommand{\HC}{\mathrm{HC}}

\newcommand{\id}{\mathrm{id}}

\DeclareMathOperator{\im}{im}

\newcommand{\opp}{{\mathrm{op}}}

\newcommand{\rmc}{\mathrm{c}}

\newcommand{\Bord}{\mathbf{Bord}}

\newcommand{\Lag}{\mathrm{Lag}}
\newcommand{\dSt}{\mathrm{d}\ms{S}\mathrm{t}}
\newcommand{\dAff}{\mathrm{d}\ms{A}\mathrm{ff}}
\newcommand{\Mod}{\ms{M}\mathrm{od}}
\newcommand{\sCAlg}{\ms{CA}\mathrm{lg}}
\newcommand{\QCoh}{\ms{QC}\mathrm{oh}}
\newcommand{\PreSymp}{\ms{P}\mathrm{re}\ms{S}\mathrm{ymp}}

\title[Higher categories of push-pull spans, II]{Higher categories of push-pull spans, II:\\Matrix factorizations}
\date{\today}
\author{Lorenzo Riva}

\begin{document}

\maketitle

\begin{abstract}
	This is the second part of a project aimed at formalizing Rozansky-Witten models in the functorial field theory framework. In the first part we constructed a symmetric monoidal $(\infty, 3)$-category $\sCRW$ of commutative Rozansky-Witten models with the goal of approximating the $3$-category of Kapustin and Rozansky. In this paper we extend work of Brunner, Carqueville, Fragkos, and Roggenkamp on the affine Rozansky-Witten models: we exhibit a functor connecting their $2$-category of matrix factorizations with the homotopy $2$-category of $\sCRW$, and calculate the associated TFTs.
\end{abstract}

\setcounter{tocdepth}{2}
\tableofcontents

\section{Introduction} \label{sec:intro}

In \cite{KRS2009} Kapustin, Rozansky, and Saulina studied the Rozansky-Witten field theory, a $3$-dimensional topological sigma model introduced in \cite{RW1997} with target space a holomorphic symplectic manifold, using path integral techniques. Their results later led them to conjecture the existence of a symmetric monoidal $3$-category $\cRW$ formed from these theories by varying the target, i.e. a $3$-category whose objects are holomorphic symplectic manifolds and whose $2$-categories of morphisms are the defects determined by the field theory -- see \cite{KR2010} for extended details and parts of the construction. Motivated by this conjecture and by the ``semi-classical'' results of \cite{CHS2022} on AKSZ models obtained using $\infty$-categorical machinery, in \cite{Riva2024} we constructed a symmetric monoidal $(\infty,3)$-category $\sCRW$  of \emph{commutative Rozansky-Witten models} with the aim of approximating $\cRW$ with the machinery of derived algebraic geometry over a field $\KK$. The construction of $\sCRW$ is general enough that it could be tweaked and applied to other geometric objects -- such as analytic or smooth stacks -- provided that they come with a good theory of symplectic structures and quasicoherent sheaves; different choices of geometric objects might lead to more accurate approximations to the desired $3$-category $\cRW$.

Another source of motivation for our previous work was the paper \cite{BCR2023} of Brunner, Carqueville, and Roggenkamp and its sequel \cite{BCFR2023} with Fragkos. In the first paper the authors describe a concrete model for a $2$-categorical truncation $\bMF$\footnote{In \cite{BCR2023} this $2$-category was simply called $\mc{C}$.} of $\cRW$ spanned by the affine planes, namely those symplectic manifolds of the form $\bT^\ast \CC^n$ for some $n \geq 0$. The $1$-morphisms in $\bMF$ can be thought of as families of Lagrangian submanifolds of $\bT^\ast \CC^n$ where each member of the family is the graph of a polynomial $1$-form on $\CC^n$, and the $2$-morphisms are chain homotopy equivalence classes of matrix factorizations associated to those polynomials. Moreover, they show that $\bMF$ has duals (namely that all objects are dualizable (\cite{BCR2023}), that all $1$-morphisms have adjoints (\cite{BCFR2023})) and explicitly exhibit all of the fully extended oriented field theories with target $\bMF$. 

If $\sCRW$ really is a decent approximation to $\cRW$ then we should expect the behavior of $\bMF$ to carry over to a $2$-categorical truncation of $\sCRW$. In this paper we confirm this prediction:
\begin{mainthm}[\Cref{thm:main}] \label{thm:mainA}
	Fix an algebraically closed field $\KK$ of characteristic $0$. Then there exists a symmetric monoidal $2$-functor
	\begin{equation*}
		\fe : \bMF \to \bh_2 \sCRW
	\end{equation*}
	with target the homotopy $2$-category of $\sCRW$. The image lands in the subcategory of $\bh_2 \sCRW$ spanned by the cotangent stacks of the form $\bT^\ast \KK^{n}$ for $n \geq 0$, and in fact $\fe$ is surjective on those objects.
\end{mainthm}

Since $\bMF$ has duals, the image of $\fe$ lands in the largest subcategory of $\bh_2 \sCRW$ with duals. Early on in the paper we prove that this is all of $\bh_2 \sCRW$:

\begin{mainthm}[\Cref{thm:2-dual-crw}] \label{thm:mainB}
	The symmetric monoidal $2$-category $\bh_2 \sCRW$ has duals for all objects and adjoints for all $1$-morphisms, so in particular every symplectic derived stack $(X, \omega) \in \bh_2 \sCRW$ is fully dualizable. 
\end{mainthm}

We can also calculate the topological field theories associated to these fully dualizable objects. The cobordism hypothesis tells us that symmetric monoidal functors
\begin{equation*}
	\cZ : \mathrm{Bord}_{0,1,2}^{\mathrm{or}} \to \bh_2 \sCRW
\end{equation*}
from the oriented cobordism $2$-category are classified by pairs $((X, \omega), \tau)$, where $\cZ(\cat{pt}^+) = (X, \omega) \in \bh_2 \sCRW$ is a fully dualizable object and $\tau : S_{(X, \omega)} \simeq \id_{(X, \omega)}$ is a trivialization of the Serre automorpshim $S_{(X, \omega)} : (X, \omega) \to (X, \omega)$. In our case there is a canonical trivialization $S_{(X, \omega)} \simeq \id_{(X, \omega)}$ (\Cref{prp:serre-trivial}), and we can use that to compute some field theories:

\begin{mainthm} \label{thm:mainC}
	Let $\cZ_A$ be the $2$-dimensional oriented topological field theory valued in $\bh_2 \sCRW$ classified by an affine symplectic derived stack $(X = \Spec A, \omega)$ with the canonical trivialization of its Serre automorphism. If for any commutative differential graded algebra $B$ we let $B^e = B \otimes B$ and $\HC(B) = B \otimes_{B^e} B$, then there are equivalences
	\begin{align*}
		\cZ_A(S^1) & \simeq \Spec \HC(A), \\
		\cZ_A(\Sigma_g) & \simeq A \otimes_{\HC(A)} (A \otimes_{A^e} \HC(\HC(A)) \otimes_{A^e} A)^{\otimes_{\HC(A)} g} \otimes_{\HC(A)} A,
	\end{align*}
	where the first value is considered as a Lagrangian morphism from the trivial stack $\ast$ to itself and the second value is considered as a $\ZZ_2$-graded vector space.
\end{mainthm}

The formulas above are consistent with the expectations in the original paper \cite{RW1997} of Rozansky and Witten, where the Rozansky-Witten sigma model was introduced. The case of $\bT^\ast \KK^n$ can be recovered by setting $A = \KK[x_1, \dotsc, x_n, p_{x_1}, \dotsc, p_{x_n}]$. It can be shown that no $\bT^\ast \KK^{n}$ is $3$-dualizable when $n > 0$, so the $2$-dimensional field theories of \Cref{thm:mainC} cannot be upgraded to fully extended $3$-dimensional field theories valued in $\sCRW$ -- see \Cref{prp:no-3-tft}.

Another consequence of the construction of $\fe$ is that the association of $2$-morphisms might not be faithful. The main problem is the failure of reconstructing a matrix factorization from its cochain complex of endomorphisms, thought of as a dg-module over a certain commutative dg-algebra. This is an interesting algebraic problem on its own for which we have no non-trivial guess as to the answer, so we pose it as a question for future consideration:
\begin{qst}
	Let $V \in \KK[\undl{x}]$ be a polynomial and let $M, N$ be two matrix factorizations of $V$. By \Cref{prp:end-is-module} the two cochain complexes $\End(M)$ and $\End(N)$ carry a canonical structure of a dg-module over a commutative differential graded algebra $A_{V}$. If $\End(M) \simeq \End(N)$ as $A_{V}$-modules, what is the relationship between $M$ and $N$?
\end{qst}

\subsubsection*{Outline and notation}

Here is an outline of the paper. In \Cref{sec:crw} we review the results of \cite{Riva2024}, some derived algebra, and we prove \Cref{thm:mainB} by standard $2$-categorical arguments applied to $\bh_2 \sCRW$ and \Cref{thm:mainC} by explicitly computing the Serre automorphism and the values of our field theories. In \Cref{sec:matrix-fact} we describe the $2$-category $\bMF$ and prove the main computational tool needed in the following section, which identifies the endomorphism algebra of the unit matrix factorization of a polynomial $V$ with the algebra of functions on the derived critical locus of the partial derivatives of $V$. Then in \Cref{sec:functor} we prove \Cref{thm:mainA} by explicitly constructing the functor. 

We use $\Mod_\KK$ to denote the $\infty$-category of cochain complexes $X$ over $\KK$, also called \emph{derived $\KK$-modules} or \emph{dg-modules} for short. If $A \in \sCAlg_\KK = \CAlg(\Mod_\KK)$ is a commutative algebra in $\Mod_\KK$, also called a \emph{commutative dg-algebra} or \emph{cdga}\footnote{Note that these are different from the \emph{curved} dgas used in \cite{KR2008}.} for short, then $\Mod_A$ denotes the $\infty$-category of modules over $A$, i.e. dg-modules $M$ equipped with an action map $A \otimes_\KK M \to M$ satisfying associativity and unitality axioms up to cohomology. We can treat any cochain complex $X$ as $\ZZ_2$-graded by summing over the even and odd degrees:
\begin{equation*}
	X \leadsto \conj{X}, \quad \conj{X}_0 = \bigoplus_{\text{even } i} X_i, \quad \conj{X}_1 = \bigoplus_{\text{odd } i} X_i.
\end{equation*}
We tend to make no notational distinction between $X$ and $\conj{X}$ and instead inform the reader when to treat an object as $\ZZ$-graded or $\ZZ_2$-graded. Then we can talk about $\sCAlg_\KK^{\ZZ_2}$, the $\infty$-category of $\ZZ_2$-graded cdgas, and $\Mod_A^{\ZZ_2}$, the $\infty$-category of $\ZZ_2$-graded derived $A$-modules. The same considerations hold for $\QCoh^{\ZZ_2}(X)$, the $\infty$-category of quasicoherent sheaves of $\ZZ_2$-graded modules over a derived stack $X$. Since everything is done in an $\infty$-categorical context we will often drop the adjective ``derived'' and just speak of $\KK$-modules, $A$-modules, and stacks.

We will try to be consistent with the notation used in \cite{Riva2024}. We use $\bh_k \sC$ to denote a homotopy $k$-category of an $(\infty,n)$-category $\sC$. We only consider the cases $k = 1$ and $k = 2$ in this paper; the first corresponds to the classical homotopy category, and the second has been worked out for complete $2$-fold Segal spaces (and hence for $n$-fold Segal spaces, $n \geq 2$) by Rom\"o in \cite{Romo2023}, and that's enough for us since $\sCRW$ was explicitly constructed as a Segal space. It can be shown that $\mathbf{h}_2 \mathscr{C}$ inherits a symmetric monoidal structure if $\mathscr{C}$ has one. The word ``equivalence'' will denote an equivalence in the appropriate $\infty$-category -- most of the time this will be a quasi-isomorphism of cochain complexes or of dg-modules over some cdga, both of which become isomorphisms in the corresponding homotopy category. We use $\circ$ for composition along a codimension $1$ boundary (e.g. composition of $1$-morphisms and vertical composition of $2$-morphisms) and $\cdot$ for composition along a codimension $2$-boundary (e.g. horizontal composition of $2$-morphisms). Algebras, commutative algebras, and modules are used in the appropriate homotopical/$\infty$-categorical sense -- see \cite{HA2017} for a comprehensive survey.

Finally, we would like to thank Nils Carqueville, Chris Schommer-Pries, and Stephan Stolz for enlightening discussions on this topic, for going through earlier versions of this paper, and for providing many helpful suggestions. We are grateful to the referees for suggesting changes in various parts of the document which greatly improved both the narrative and the overall clarity of the paper. 

\section{Brief review of $\sCRW$ and dualizability results} \label{sec:crw}

Throughout this paper we assume that $\KK$ is an algebraically closed field of characteristic $0$. Geometrically-minded readers might find some comfort in taking $\KK = \CC$, but they should be warned that the analytical properties of $\CC$ do not carry over to this context -- it is all and only algebra from now on. In this section we will review the main results of the previous paper, prove that $\bh_2 \sCRW$ has duals, and then define the symplectic derived stacks which we will need in the next section.

\subsection{The basics}

Recall Theorem B of \cite{Riva2024}, from which we know that there exists a symmetric monoidal $(\infty,3)$-category $\sCRW$ with the following properties:
\begin{enumerate}
	\item its objects are symplectic derived stacks $(X, \omega_X)$ over $\KK$ (see \cite{PTVV2013}, for example);
	\item its $1$-morphisms $X \to Y$ are Lagrangian spans $X \leftarrow L \to Y$ (see \cite{Calaque2015}, for example), which compose via (derived) pullbacks;
	\item the $2$-morphisms $L \Rightarrow L'$ form the $\infty$-category $\QCoh^{\ZZ_2}(L \cap L')$ of $\ZZ_2$-graded quasi-coherent sheaves on the derived intersection $L \cap L' := L \times_{X \times Y} L'$; these compose via a push-pull formula (see \cite[Section 1.2.1]{Riva2024});
	\item the symmetric monoidal product is given by
	\begin{equation*}
		(X, \omega_X) \odot (Y, \omega_Y) = (X \times Y, \pi_X^\ast \omega_X + \pi_Y^\ast \omega_Y)
	\end{equation*}
	where $\pi_X$ and $\pi_Y$ are the canonical projections from $X \times Y$.
\end{enumerate}

The structure of the $2$- and $3$-morphisms, and especially their composition, is complicated. Fortunately, the bulk of this paper will only require us to work with a concrete class of derived stacks, namely those which are \emph{affine}. An arbitrary derived stack over $\KK$ is a sheaf of spaces on $\sCAlg^\rmc_{\KK}$, the $\infty$-category of connective (i.e. with vanishing positive cohomology) cdgas over $\KK$, satisfying descent for \'etale morphisms. The affine derived stacks are those which are completely determined by a single cdga: there is a fully faithful functor (the Yoneda embedding)
\begin{equation*}
	\Spec : (\sCAlg_\KK^{\rmc})^\opp =: \dAff \to \dSt_\KK
\end{equation*}
and we define the subcategory of affine stacks
\begin{equation*}
	\dAff_\KK := \im \Spec \simeq (\sCAlg_\KK^{\rmc})^\opp
\end{equation*}
to be the essential image of $\Spec$. In particular we have
\begin{equation*}
	\Map_{\dSt_\KK}(\Spec A, \Spec B) \simeq \Map_{\sCAlg_\KK^{\rmc}}(B, A)
\end{equation*}
for any $A, B \in \sCAlg_\KK^\rmc$. The advantage of working with affine stacks is that maps of cdgas can be constructed directly if one has an explicit model for them, and this will be essential to us in \Cref{sec:bMF} and \Cref{sec:functor}. Note also that all of the examples of cdgas in the paper are obtained from polynomial algebras or exterior algebras with easily defined concrete differentials, so that maps between them are simply polynomial maps in a finite number of variables.

\subsection{$2$-dualizability in $\sCRW$}

In this subsection we will prove a dualizability result which will help us calculate some field theories later on.

\begin{thm} \label{thm:2-dual-crw}
	Every object $X$ of $\sCRW$ is $2$-dualizable: there is a dual $X^\diamond$ with (co)evaluation maps $\eta : \ast \to X \odot X^\diamond$ and $\veps : X^\diamond \odot X \to \ast$, and $\eta$ and $\veps$ are part of an infinite chain of left and right adjoints:
	\begin{gather*}
		\dotsb \dashv \eta^{LL} \dashv \eta^L \dashv \eta \dashv \eta^R \dashv \eta^{RR} \dashv \dotsb, \\
		\dotsb \dashv \veps^{LL} \dashv \veps^L \dashv \veps \dashv \veps^R \dashv \veps^{RR} \dashv \dotsb.
	\end{gather*}
\end{thm}

It will be enough to prove the existence of the dual $X^\diamond$ (together with $\eta$ and $\veps$) and then prove that all $1$-morphisms in $\sCRW$ have both left and right adjoints. We start with duals:

\begin{prp} \label{prp:duals}
	The dual $(X, \omega)^\diamond$ of a derived symplectic stack $(X, \omega)$ is $(X, -\omega)$, with (co)evaluation maps given by the diagonal $\Delta_X$ read as a Lagrangian span between $\ast$ and $X \times X$.
\end{prp}

\begin{proof}
	By construction, the symmetric monoidal homotopy $1$-category $\bh_1 \sCRW$ admits a symmetric monoidal functor from $\bh_1 \Lag_{(\infty,1)}^0$, the homotopy $1$-category of the $(\infty,1)$-category of ($0$-shifted) symplectic derived stacks and Lagrangian correspondences defined in \cite[Section 14]{Haugseng2018}. This functor is the identity on objects and sends equivalence classes of Lagrangian spans to ``quasicoherent Morita classes'' of Lagrangian spans: two Lagrangian spans $L, M$ between $X$ and $Y$ are equivalent if there are quasicoherent sheaves $Q, Q' \in \QCoh^{\ZZ_2}(L \times_{X \times Y} M)$ such that
	\begin{equation*}
		\pi_{L,L, \ast} (\pi_{L,M}^\ast Q \otimes \pi_{M,L}^\ast Q') \simeq \cO_{L \times_{X \times Y} L}, \quad \pi_{M,M, \ast} (\pi_{M,L}^\ast Q' \otimes \pi_{L,M}^\ast Q) \simeq \cO_{M \times_{X \times Y} M},
	\end{equation*}
	where the $\pi$'s are the appropriate projections from the triple products $L \times M \times L$ and $M \times L \times M$ -- see the push-pull formula in \cite{Riva2024}. This equivalence relation is coarser than the usual equivalence of spans: if $L$ and $ M$ are equivalent spans between $X$ and $Y$ (so in particular $L \simeq M$) then we can pick $Q = Q' = i_\ast \cO_{L} \simeq i_\ast \cO_{M}$, the pushforward of the structure sheaf on the diagonal $L$ along the map $i : L \to L \times_{X \times Y} L \simeq L \times_{X \times Y} M$. Therefore the functor is well-defined, as claimed.
	
	Since symmetric monoidal functors carry dualizability data to dualizability data and since dualizable objects of a symmetric monoidal $\infty$-category are detected within its homotopy category, the result will follow if every object of $\Lag_{(\infty,1)}^0$ is dualizable. This is proven in \cite[Proposition 14.16]{Haugseng2018}, so we're done.
\end{proof}

\begin{prp} \label{prp:adjoints}
	If $f = (X \leftarrow L \rightarrow Y)$ is a Lagrangian span representing a $1$-morphism in $\sCRW$ then $f^\dagger = (Y \leftarrow L \rightarrow X)$, the same span but read backwards, is canonically both a left and right adjoint for $f$.
\end{prp}

\begin{proof}
	We will only show that $f$ is left adjoint to $f^\dagger$; the other claim follows since we have an equivalence of spaces
	\begin{equation*}
		\Map_{\sCRW}(f,f') \simeq \QCoh^{\ZZ_2}(L \cap L')^\simeq \simeq \QCoh^{\ZZ_2}(L' \cap L)^\simeq \simeq \Map_{\sCRW}(f',f)
	\end{equation*}
	(with $\sC^\simeq$ denoting the core, or the space of objects, of an $\infty$-category $\sC$) for any $1$-morphism $f' = (X \leftarrow L' \rightarrow Y)$, so the unit and counit maps for one adjunction can be reflected to become the counit and unit maps for the opposite adjunction.
	
	One of the equivalent definition of an adjunction $f \dashv f^\dagger$ is that there is an equivalence of mapping spaces
	\begin{equation*}
		\Map_{\sCRW}(f \circ g, h) \simeq \Map_{\sCRW}(g, f^\dagger \circ h)
	\end{equation*}
	natural in $g$ and $h$, where $g$ is a Lagrangian span $W \leftarrow K \rightarrow X$ and $h$ is a Lagrangian span $W \leftarrow M \rightarrow Y$. Indeed, there are natural equivalences
	\begin{equation*}
		\QCoh^{\ZZ_2}((K \times_X L) \times_{W \times Y} M) \simeq \QCoh^{\ZZ_2}(P) \simeq \QCoh^{\ZZ_2}(K \times_{W \times X} (L \times_Y M))
	\end{equation*}
	of $\infty$-categories (and hence of their cores) since both derived stacks $(K \times_X L) \times_{W \times Y} M$ and $K \times_{W \times X} (L \times_Y M)$ are canonically and uniquely identified with the limit $P$ of the following diagram of spans:
	\begin{equation*}
		\begin{tikzcd}
			& Y & \\
			M \ar[ur] \ar[d] & & L \ar[ul] \ar[d] \\
			W & & X \\
			& K \ar[ul] \ar[ur] &
		\end{tikzcd}
	\end{equation*}
	This concludes the proof.
\end{proof}

The equivalence
\begin{equation*}
	\Map_{\sCRW}(f \circ g, h) \simeq \Map_{\sCRW}(g, f^\dagger \circ h)
\end{equation*}
can be used to compute the (co)units of the adjunction $f \dashv f^\dagger$. To do so we set $g = \id_X$ and $h = f$: then the unit $\eta$ is obtained as the image of $\id_f$ under the equivalence, namely
\begin{equation*}
	\Map_{\sCRW}(f, f) \simeq \Map_{\sCRW}(\id_X, f^\dagger \circ f), \quad \id_f \mapsto \eta.
\end{equation*}
Since $\id_f = i_\ast \cO_{L}$, the pushforward of the structure sheaf on $L$ along the inclusion $i : L \to L \times_{X \times Y} L$ of the diagonal of $L$, we see that $\eta$ must be the same sheaf $i_\ast \cO_{L}$ over the same stack $L \times_{X \times Y} L$ but considered as a $2$-morphism $\id_X \to f^\dagger \circ f$. Similarly the counit $\veps$ is $i_\ast \cO_L$, this time seen as a $2$-morphism $f \circ f^\dagger \to \id_Y$.

\subsection{The affine part} \label{sec:affine-part}

It will be helpful to spell out the various products, compositions, and units in $\sCRW$ in the case where everything is affine. In particular we have
\begin{equation*}
	\QCoh^{\ZZ_2}(\Spec A) \simeq \Mod_A^{\ZZ_2};
\end{equation*}
see the construction of $\QCoh^{\ZZ_2}$ in \cite[Section 3.1.3]{Riva2024}. We omit the mention of the symplectic forms for simplicity, but the reader should know that every diagram below comes equipped with the extra data of those forms, together with various nullhomotopies when discussing the Lagrangian condition.

If we have two composable spans
\begin{equation*}
	\begin{tikzcd}[column sep = tiny]
		& \Spec R \ar[dl] \ar[dr] & & & \Spec R' \ar[dl] \ar[dr] \\
		\Spec A & & \Spec B & \Spec B & & \Spec C
	\end{tikzcd}
\end{equation*}
then their composite is given by the derived pullback over $\Spec B$, or equivalently by the span
\begin{equation*}
	\begin{tikzcd}[column sep = tiny]
		& \Spec (R \otimes_B R') \ar[dl] \ar[dr] & \\
		\Spec A & & \Spec C.
	\end{tikzcd}
\end{equation*}
It follows from \cite[Theorem 4.4]{Calaque2015} that this span is again Lagrangian. The unit for this composition is the diagonal span
\begin{equation*}
	\begin{tikzcd}[column sep = tiny]
		& \Spec A \ar[dl, equal] \ar[dr, equal] & \\
		\Spec A & & \Spec A.
	\end{tikzcd}
\end{equation*}
both of whose legs are the identity map. In terms of the underlying algebra map of the morphism $\Spec A \to \Spec A \times \Spec A$, this span is given by the multiplication map $A \otimes_\KK A \to A$. 

If we have two parallel Lagrangian spans
\begin{equation*}
	\begin{tikzcd}[row sep = small, column sep = tiny]
		& \Spec R \ar[dl] \ar[dr] & \\
		\Spec A & & \Spec B \\
		& \Spec S \ar[ul] \ar[ur] & 
	\end{tikzcd}
\end{equation*}
then their derived intersection is $\Spec (R \otimes_{AB} S)$, where $AB = A \otimes_\KK B$, and thus the $\infty$-category of $2$-morphisms between them is 
\begin{equation} \label{eqn:cat-of-2-morphisms}
	\QCoh^{\ZZ_2}(\Spec (R \otimes_{AB} S)) \simeq \Mod^{\ZZ_2}_{R \otimes_{AB} S}.
\end{equation}

The vertical push-pull composition is given as follows: let $X$ and $Y$ be two composable $2$-morphisms:
\begin{equation*}
	\begin{tikzcd}[column sep = small]
		& \Spec R \ar[d, "X", Rightarrow] \ar[dl] \ar[dr] & \\
		\Spec A & \Spec S \ar[d, "Y", Rightarrow] \ar[l] \ar[r] & \Spec B \\
		& \Spec T \ar[ul] \ar[ur] &
	\end{tikzcd}
\end{equation*}
This means that $X \in \Mod^{\ZZ_2}_{R \otimes_{AB} S}$ and $Y \in \Mod^{\ZZ_2}_{S \otimes_{AB} T}$. Then their composite is obtained by first passing the modules $X$ and $Y$ to the triple intersection
\begin{equation*}
	\Spec R \cap \Spec S \cap \Spec T \simeq \Spec (R \otimes_{AB} S \otimes_{AB} T)
\end{equation*}
via the functors
\begin{equation*}
	- \otimes_{R \otimes_{AB} S} (R \otimes_{AB} S \otimes_{AB} T) \quad \text{and} \quad - \otimes_{S \otimes_{AB} T} (R \otimes_{AB} S \otimes_{AB} T),
\end{equation*}
then tensoring the resulting modules together, and finally restricting along the map $R \otimes_{AB} T \to R \otimes_{AB} S \otimes_{AB} T$. After some simplifications we see that the composite is
\begin{multline*}
	(X \otimes_{R \otimes_{AB} S} (R \otimes_{AB} S \otimes_{AB} T)) \otimes_{R \otimes_{AB} S \otimes_{AB} T} (Y \otimes_{S \otimes_{AB} T} (R \otimes_{AB} S \otimes_{AB} T) \\
	\simeq (X \otimes_{{AB}} T) \otimes_{R \otimes_{AB} S \otimes_{AB} T} (Y \otimes_{AB} R) \simeq X \otimes_S Y
\end{multline*}
considered as an $(R \otimes_{AB} T)$-module. The unit for this composition is $S \in \Mod^{\ZZ_2}_{S \otimes_{{AB}} S}$ considered as an algebra over $S \otimes_{{AB}} S$ via the multiplication map.

The horizontal composition is given as follows: let $X$ and $X'$ be two $2$-morphism sharing a $0$-dimensional boundary component:
\begin{equation*}
	\begin{tikzcd}[row sep = small, column sep = small]
		& \Spec R \ar[dd, "X", Rightarrow] \ar[dl] \ar[dr] & & \Spec R' \ar[dd, "X'", Rightarrow] \ar[dl] \ar[dr] & \\
		\Spec A & & \Spec B & & \Spec C \\
		& \Spec S \ar[ul] \ar[ur] & & \Spec S' \ar[ul] \ar[ur] &
	\end{tikzcd}
\end{equation*}
Then their composite is obtained first by passing to the intersection of the pullbacks
\begin{align*}
	(\Spec R \times_{\Spec B} \Spec R') & \cap (\Spec S \times_{\Spec B} \Spec S') \\
	& \simeq \Spec (R \otimes_B R') \cap \Spec (S \otimes_B S') \\
	& \simeq \Spec ((R \otimes_B R') \otimes_{AC} (S \otimes_B S')) \\
	& \simeq \Spec ((R \otimes_{AB} S) \otimes_{B} (R' \otimes_{BC} S'))
\end{align*}
and then by tensoring. After some simplifications we see that the composite is
\begin{equation*}
	(X \otimes_B (R' \otimes_{BC} S')) \otimes_{(R \otimes_{AB} S) \otimes_{B} (R' \otimes_{BC} S')} (X' \otimes_B (R \otimes_{AB} S)) \simeq X \otimes_B X'
\end{equation*}
in the category of $((R \otimes_{AB} S) \otimes_{B} (R' \otimes_{BC} S'))$-modules. The unit for this composition is $B \in \Mod^{\ZZ_2}_{B}$.

\subsection{$2$-dimensional TFTs associated to affine symplectic derived stacks}

By \Cref{thm:2-dual-crw} we know that every symplectic derived stack $(X, \omega)$ is $2$-dualizable, i.e. fully dualizable in $\bh_2 \sCRW$. Each one of these determines a family of $2$-dimensional oriented topological field theories
\begin{equation*}
	\cZ_{(X, \omega), \tau} : \mathrm{Bord}_{0,1,2}^{\mathrm{or}} \to \bh_2 \sCRW
\end{equation*}
parametrized by the trivializations $\tau : S_{(X, \omega)} \simeq \id_{(X, \omega)}$ of the Serre automorphism, which is a $1$-morphism arising from the evaluation, coevaluation, and braiding morphisms of $(X, \omega)$. This can be seen by combining two statements:
\begin{enumerate}
	\item the oriented extended $2$-dimensional cobordism hypothesis
	\begin{equation*}
		\Fun^\otimes(\Bord_{0,1,2}^{\mathrm{or}}, \sC) \simeq (\sC^{\mathrm{fd}})^{\mathrm{SO}(2)},
	\end{equation*}
	which classifies functors out of $\Bord_{0,1,2}^{\mathrm{or}}$ as (homotopy) fixed points of the $\mathrm{SO}(2)$-action on the space of fully dualizable objects in $\sC$;
	\item the computation of \cite[Corollary 4.4]{HV2019} which provides an identification
	\begin{equation*}
		(\sC^{\mathrm{fd}})^{\mathrm{SO}(2)} \simeq \{\text{fully dualizable object $x \in \sC$ } + \text{ trivialization $\lambda_x : S_x \cong \id_x$}\}. 
	\end{equation*}
\end{enumerate}
We direct the reader to \cite{SP2009} for an extended treatment of both the unoriented and the oriented extended $2$-dimensional cobordism hypothesis, and to \cite{Pstragowski2014,Pstragowski2022} for another proof of the oriented case and a considerable simplification of the dualizability data. In particular, thanks to the work of Pstragowski we are guaranteed that several coherence identities that relate the higher adjoints of a fully dualizable object (the \emph{swallowtail identities}, for example) are automatically satisfied up to replacing some (co)unit morphisms by isomorphic ones. We also refer the reader to \cite[Section 2.3]{BCFR2023} for more details about these equations and especially for the accompanying pictures, which are very helpful when trying to understand how to relate the algebra with the cobordisms.

In this subsection we will prove that $S_{(X, \omega)}$ admits a canonical equivalence $\tau^\mathrm{can} : S_{(X, \omega)} \simeq \id_{(X, \omega)}$ which we can use as our preferred trivialization. Then we will compute the values of $\cZ_{(X, \omega), \tau^{\mathrm{can}}}$, when $X = \Spec A$ is affine, on all closed manifolds of dimension up to $2$. Note that we will denote this field theory by $\cZ_A$ since its values on closed $1$- and $2$-manifolds only depend on $A$ and not on $\omega$, as can be seen from the computations.

The process for computing the field theory with an arbitrary, non-affine $X$ is completely analogous but notationally challenging, with one having to replace the local algebras of functions with their global sheaf counterparts and the tensor products and restrictions with pullbacks and pushforwards. We encourage the reader to work out one example for themselves to see the complexity of the final expression and compare it to the affine case. Moreover, as we will see, the field theories induced from the functor that we will construct in \Cref{sec:functor} are associated to such affine stacks, which is another reason to concentrate on this case.

\subsubsection{The Serre automorphism}

Let $\mathbf{C}$ be a symmetric monoidal $2$-category, let $p \in \mathbf{C}$ be a fully dualizable object, let $b : p \otimes p \to p \otimes p$ be the braiding isomorphism of $p$, and let $e : p^\vee \otimes p \to 1$ be the evaluation morphism of $p$ with right adjoint $e^\dagger : 1 \to p^\vee \otimes p$. Then the Serre automorphism is the invertible map
\begin{equation*}
	S_p := (\id_p \otimes e) \circ (b \otimes \id_{p^\vee}) \circ (\id_p \otimes e^\dagger).
\end{equation*}
Pictorially, the Serre automorphism corresponds to a copy of the interval with a twist in the middle and its trivializations correspond to ways of tightening up the twist until it becomes a straight line -- see \Cref{fig:serre}.
\begin{figure}[ht]
	\centering
	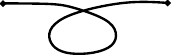
	\hspace{4em}
	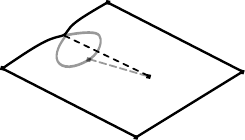
	\caption{The Serre automorphism, on the left, and a way to ``tighten the twist'' to a straight line, on the right.}
	\label{fig:serre}
\end{figure}

In our case $\mathbf{C} = \bh_2 \sCRW$ and $p = (X, \omega)$ is any symplectic derived stack, and we have:

\begin{prp} \label{prp:serre-trivial}
	The Serre automorphism in $\bh_2 \sCRW$ admits a canonical trivialization $S_{(X, \omega)} \simeq \id_{(X, \omega)}$.
\end{prp}

\begin{proof}
	The formula for the Serre automorphism applied to $p = (X, \omega)$ amounts to taking the limit of the following diagram of derived stacks:
	\begin{equation*}
		\begin{tikzcd}[column sep = small]
			& X \times X \ar[dl, "\pi_1"'] \ar[dr, "\id_X \times i"] & & X \times X \times X \ar[dl, equal] \ar[dr, "\sigma \times \id_X"] & & X \times X \ar[dl, "\id_X \times i"'] \ar[dr, "\pi_1"] & \\
			X & & X \times X \times X & & X \times X \times X & & X
		\end{tikzcd}
	\end{equation*}
	where $\pi_1$ is the projection on the first factor, $i$ is the diagonal map, and $\sigma$ is the map swapping the two factors in the product. By an argument completely analogous to that in the proof of \cite[Lemma 12.3]{Haugseng2018}, the limit of this diagram is canonically identified with $X$ and the two projection maps, under that identification, are the identity. This represents the identity span among symplectic derived stacks, so $S_{(X, \omega)} \simeq \id_{(X, \omega)}$, as needed.
\end{proof}

Now trivializations of $S_{(X, \omega)}$ are precisely the automorphisms of $\id_{(X, \omega)}$, namely invertible elements of $\QCoh^{\ZZ_2}(X \times_{X \times X} X)$. If $X = \Spec A$ is affine then 
\begin{equation*}
	\QCoh^{\ZZ_2}(X \times_{X \times X} X) \simeq \Mod^{\ZZ_2}_{A \otimes_{A \otimes A} A}
\end{equation*}
and so the trivializations correspond to invertible $A \otimes_{A \otimes A} A$-modules. One example is $A \otimes_{A \otimes A} A$ itself considered as a $A \otimes_{A \otimes A} A$-module via the multiplication map, but there might be many other non-trivial ones. From now on, however, we will concentrate on the case where our trivialization is the canonical one from \Cref{prp:serre-trivial}, and we let $\cZ_A$ denote the unique $2$-dimensional oriented field theory associated to $(\Spec A, \omega)$ with that trivialization. 

\subsubsection{The $0$- and $1$-dimensional manifolds} 

From \Cref{prp:duals} we know that $\cZ_A(\cat{pt}^-) = (X, -\omega)$. If $\cat{coev} : \emptyset \to \cat{pt}^+ \sqcup \cat{pt}^-$ and $\cat{ev} : \cat{pt}^- \sqcup \cat{pt}^+ \to \emptyset$ denote the left and right elbow bordisms (see \Cref{fig:elbows}) we get that $\cZ_A(\cat{coev})$ and $\cZ_A(\cat{ev})$ are given by reading the Lagrangian span
\begin{equation*}
	\ast \leftarrow X \xlra{i} X \times X,
\end{equation*}
where $i$ is the diagonal map, from left to right for the coevaluation and from right to left for the evaluation. 

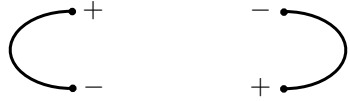
\begin{figure}[h!]
	\centering
\begingroup%
  \makeatletter%
  \providecommand\color[2][]{%
    \errmessage{(Inkscape) Color is used for the text in Inkscape, but the package 'color.sty' is not loaded}%
    \renewcommand\color[2][]{}%
  }%
  \providecommand\transparent[1]{%
    \errmessage{(Inkscape) Transparency is used (non-zero) for the text in Inkscape, but the package 'transparent.sty' is not loaded}%
    \renewcommand\transparent[1]{}%
  }%
  \providecommand\rotatebox[2]{#2}%
  \newcommand*\fsize{\dimexpr\f@size pt\relax}%
  \newcommand*\lineheight[1]{\fontsize{\fsize}{#1\fsize}\selectfont}%
  \ifx\svgwidth\undefined%
    \setlength{\unitlength}{45.83203053bp}%
    \ifx\svgscale\undefined%
      \relax%
    \else%
      \setlength{\unitlength}{\unitlength * \real{\svgscale}}%
    \fi%
  \else%
    \setlength{\unitlength}{\svgwidth}%
  \fi%
  \global\let\svgwidth\undefined%
  \global\let\svgscale\undefined%
  \makeatother%
  \begin{picture}(1,0.81695659)%
    \lineheight{1}%
    \setlength\tabcolsep{0pt}%
    \put(0,0){\includegraphics[width=\unitlength,page=1]{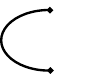}}%
    \put(0.6032344,0.66584434){\color[rgb]{0,0,0}\makebox(0,0)[lt]{\lineheight{1.25}\smash{\begin{tabular}[t]{l}$+$\end{tabular}}}}%
    \put(0.6094975,0.03461384){\color[rgb]{0,0,0}\makebox(0,0)[lt]{\lineheight{1.25}\smash{\begin{tabular}[t]{l}$-$\end{tabular}}}}%
  \end{picture}%
\endgroup%

	\hspace{4em}
\begingroup%
  \makeatletter%
  \providecommand\color[2][]{%
    \errmessage{(Inkscape) Color is used for the text in Inkscape, but the package 'color.sty' is not loaded}%
    \renewcommand\color[2][]{}%
  }%
  \providecommand\transparent[1]{%
    \errmessage{(Inkscape) Transparency is used (non-zero) for the text in Inkscape, but the package 'transparent.sty' is not loaded}%
    \renewcommand\transparent[1]{}%
  }%
  \providecommand\rotatebox[2]{#2}%
  \newcommand*\fsize{\dimexpr\f@size pt\relax}%
  \newcommand*\lineheight[1]{\fontsize{\fsize}{#1\fsize}\selectfont}%
  \ifx\svgwidth\undefined%
    \setlength{\unitlength}{36.12492299bp}%
    \ifx\svgscale\undefined%
      \relax%
    \else%
      \setlength{\unitlength}{\unitlength * \real{\svgscale}}%
    \fi%
  \else%
    \setlength{\unitlength}{\svgwidth}%
  \fi%
  \global\let\svgwidth\undefined%
  \global\let\svgscale\undefined%
  \makeatother%
  \begin{picture}(1,1.03926454)%
    \lineheight{1}%
    \setlength\tabcolsep{0pt}%
    \put(0,0){\includegraphics[width=\unitlength,page=1]{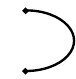}}%
    \put(-0.01883293,0.847547){\color[rgb]{0,0,0}\makebox(0,0)[lt]{\lineheight{1.25}\smash{\begin{tabular}[t]{l}$-$\end{tabular}}}}%
    \put(-0.01885568,0.04391505){\color[rgb]{0,0,0}\makebox(0,0)[lt]{\lineheight{1.25}\smash{\begin{tabular}[t]{l}$+$\end{tabular}}}}%
  \end{picture}%
\endgroup%

	\caption{$\cat{coev}$ on the left and $\cat{ev}$ on the right.}
	\label{fig:elbows}
\end{figure}

This span also presents $\cZ_A(\cat{coev})$ as the (left and right) adjoint to $\cZ_A(\cat{ev})$ (see \Cref{prp:adjoints}) meaning that up to a canonical braiding isomorphism we can replace all instances of $\cZ_A(\cat{ev}^\dagger)$ (where the symbol $\dagger$ can mean either left or right adjoint) with $\cZ_A(\cat{coev})$, and vice versa for $\cZ_A(\cat{coev})$. We can immediately calculate the value of $\cZ_A$ on the circle:
\begin{equation*}
	\cZ_A(S^1) = \cZ_A(\cat{ev}) \circ \cZ_A(\cat{coev}) \simeq X \times_{X \times X} X \simeq \Spec (A \otimes_{A \otimes A} A) \simeq \Spec \HC(A)
\end{equation*}
as a span of derived stacks between two copies of the trivial stack $\ast$. Here $\HC(A)$ denotes the Hochschild complex of $A$, which is definitionally the derived tensor product $A \otimes_{A \otimes A} A$. We make the following abbreviations for future use: $A^e := A \otimes A$ and $H := \HC(A)$, so that $H = A \otimes_{A^e} A$.

\subsubsection{The $2$-dimensional manifolds.} The values of the four basic cobordisms -- the saddle, the cosaddle, the cap, and the cocap (or cup), see \Cref{fig:2-bords} -- are either units or counits of one of the adjunctions $\cZ_A(\cat{ev}) \dashv \cZ_A(\cat{coev})$ and $\cZ_A(\cat{coev}) \dashv \cZ_A(\cat{ev})$. 

\begin{figure}[h!]
	\centering
	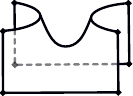
	\hspace{2em}
	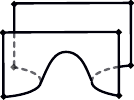
	\hspace{2em}
	\rotatebox{90}{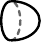}
	\hspace{2em}
	\rotatebox{90}{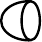}
	\caption{From left to right: $\cat{sad}$, $\cat{cosad}$, $\cat{cap}$, and $\cat{cocap}$. We omitted the $+$ and $-$ distinguishing the framing on the points since applying $\cZ_A$ nullifies the distinction due to the ambidexterity of adjoints in this $2$-category.}
	\label{fig:2-bords}
\end{figure}

More specifically:
\begin{enumerate}
	\item[($\eta$)] $\cZ_A(\cat{sad}) : \cZ_A(\cat{coev}) \circ \cZ_A(\cat{ev}) \Rightarrow \cZ_A(\id_{\cat{pt}^+ \sqcup \cat{pt}^-})$ and $\cZ_A(\cat{cap}) : \cZ_A(\emptyset) \Rightarrow \cZ_A(S^1)$ are the units of $\cZ_A(\cat{ev}) \dashv \cZ_A(\cat{coev})$ and $\cZ_A(\cat{coev}) \dashv \cZ_A(\cat{ev})$, respectively;
	\item[($\veps$)] $\cZ_A(\cat{cosad}) : \cZ_A(\id_{\cat{pt}^+ \sqcup \cat{pt}^-}) \Rightarrow \cZ_A(\cat{coev}) \circ \cZ_A(\cat{ev})$ and $\cZ_A(\cat{cocap}) : \cZ_A(S^1) \Rightarrow \cZ_A(\emptyset)$ are the counits of $\cZ_A(\cat{coev}) \dashv \cZ_A(\cat{ev})$ and $\cZ_A(\cat{ev}) \dashv \cZ_A(\cat{coev})$, respectively.
\end{enumerate}
From \Cref{prp:adjoints} and the subsequent discussion we know that these values are given by a single module (or, more generally, sheaf) with differently labelled boundaries. The module is $\cZ_A(\id_{\cat{ev}}) = A$ as an $H$-module via the canonical fold map $H \to A$. To distinguish them during compositions we will use the notation
\begin{equation*}
	A^{\cat{sad}} = \cZ_A(\cat{sad}), \, A^{\cat{cap}} = \cZ_A(\cat{cap}), \,  A^{\cat{cosad}} = \cZ_A(\cat{cosad}), \, A^{\cat{cocap}} = \cZ_A(\cat{cocap})
\end{equation*}
where each of those is the same $H$-module $A$. We will also set $A^{\cat{ev}} = \cZ_A(\id_\cat{ev})$ and $A^{\cat{coev}} = \cZ_A(\id_\cat{coev})$. Then the value of the pants, seen as a bordism $S^1 \sqcup S^1 \Rightarrow S^1$, is
\begin{equation*}
	\cZ_A(\cat{pant}) \simeq \cZ_A(\id_{\cat{ev}}) \cdot \cZ_A(\cat{sad}) \cdot \cZ_A(\id_{\cat{coev}}) \simeq A^{\cat{ev}} \otimes_{A^e} A^{\cat{sad}} \otimes_{A^e} A^{\cat{coev}}
\end{equation*}
as an $H \otimes H \otimes H$-module, and the same formula holds for the copants with $\cat{sad}$ replaced by $\cat{cosad}$. These formulas imply that the twice-punctured torus gets assigned the value value
\begin{align*}
	\cZ_A(\cat{copant}) \circ \cZ_A(\cat{pant}) & \simeq (A^{\cat{ev}} \otimes_{A^e} A^{\cat{cosad}} \otimes_{A^e} A^{\cat{coev}}) \\
	& \qquad \otimes_{H \otimes H} (A^{\cat{ev}} \otimes_{A^e} A^{\cat{sad}} \otimes_{A^e} A^{\cat{coev}}) \\
	& \simeq A \otimes_{A^e} (A \otimes_{A^e} A) \otimes_{H \otimes H} (A \otimes_{A^e} A) \otimes_{A^e} A \\
	& \simeq A \otimes_{A^e} \HC(H) \otimes_{A^e} A
\end{align*}
as an $H \otimes H$-module. We abbreviate $\HC(H) = \HC(\HC(A))$ as $H^2$. Thus the genus $g$ surface is sent to the $\KK$-module
\begin{align*}
	\cZ_A(\Sigma_g) & \simeq \cZ_A(\cat{cocap}) \circ (\cZ_A(\cat{copant}) \circ \cZ_A(\cat{pant}))^{\circ g} \circ \cZ_A(\cat{cap}) \\
	& \simeq A \otimes_{H} (A \otimes_{A^e} H^2 \otimes_{A^e} A)^{\otimes_{H} g} \otimes_{H} A,
\end{align*}
where the superscripts ``$\circ g$'' and ``$\otimes_H g$'' denotes the $g$-fold composition and tensor product, respectively, of an object with itself. This computation, together with $\cZ_A(S^1) \simeq \Spec \HC(A)$, concludes the proof of \Cref{thm:mainC}.

\subsubsection{Failure of $3$-dualizability}

We close this subsection on a negative, but expected, note. Suppose that the $2$-dimensional field theory $\cZ_A$ that we just calculated could be extended to a $3$-dimensional fully extended field theory $\conj{\cZ}$ valued in $\sCRW$, i.e. 
\begin{equation*}
	\conj{\cZ} : \Bord_{0,1,2,3}^{\mathrm{or}} \to \sCRW
\end{equation*}
is a symmetric monoidal functor of $(\infty,3)$-categories such that 
\begin{equation*}
	\bh_2 \conj{\cZ} = \conj{\cZ}\vert_{\mathrm{Bord}_{0,1,2}^{\mathrm{or}}} \simeq \cZ_A.
\end{equation*}
Then $\conj{\cZ}$ could be dimensionally reduced to a $1$-dimensional field theory
\begin{equation*}
	\conj{\cZ}^{\mathrm{red}} = \conj{\cZ}(- \times S^2) : \mathrm{Bord}^{\mathrm{or}}_{0,1} \to \End(\id_\ast) \simeq \Mod_{k}^{\ZZ_2}.
\end{equation*}
It is a classical result that such field theories are classified by the $\ZZ_2$-graded cochain complexes which are finite dimensional in both even and odd degree. This does not happen in our cases of interest:

\begin{prp} \label{prp:no-3-tft}
	If $A$ is a polynomial algebra concentrated in degree $0$ with at least $1$ generator, then $\cZ_A$ cannot be extended to a $3$-dimensional field theory.
\end{prp}

\begin{proof}
	A quick calculation shows that if $A = \KK[x_1, \dotsc, x_t]$ then
	\begin{align*}
		\conj{\cZ}^{\mathrm{red}}(\cat{pt}^+) = \cZ_A(S^2) & \simeq A \otimes_{\HC(A)} A \\
		& \simeq \KK[x_1, \dotsc, x_t, y_1, \dotsc, y_t] \otimes_\KK \Lambda(x_1, \dotsc, x_t, y_1, \dotsc, y_t)
	\end{align*}
	is an algebra generated by $2t$ generators in degree $0$ and $2t$ generators in degree $-1$, with no differential. (Here $\Lambda(\dotsc)$ denotes the exterior algebra on a set of generators.) So $\cZ(S^2)$ is never a finite dimensional $\ZZ_2$-graded vector space unless $t = 0$, in which case $A = \KK$ and $\Spec A = \ast$ is the unit in $\sCRW$.
\end{proof}
We will show later that the cotangent stack of $\KK^n$ is a symplectic derived stack and satisfies 
\begin{equation*}
	\bT^\ast \KK^n \simeq \Spec \KK[x_1, \dotsc, x_n, p_{x_1}, \dotsc, p_{x_n}].
\end{equation*}
Therefore, by the proposition, the field theories with target $\bT^\ast \KK^n$ for $n \geq 0$ are purely $2$-dimensional and cannot be extended further.

\subsection{Aside: cotangent stacks and affine planes} \label{sec:affine-planes}

Before moving on to the next section we need to define one of the principal objects of interest for our discussion, namely the \emph{affine plane} $\bT^\ast \KK^n$. This is the cotangent stack of the affine $n$-dimensional space $\KK^n = \Spec \KK[x_1, \dotsc, x_n]$. In this last subsection we will review some of the theory of cotangent stacks and the symplectic forms that can be put on them and then specialize to $\bT^\ast \KK^n$.

Fix a derived stack $X$ and consider its \emph{cotangent stack} $\bT^\ast X \in \dSt_{/X}$, defined up to equivalence by
\begin{equation*}
	\dSt_{/X}(Y, \bT^\ast X) \simeq \QCoh(Y)(\cO_Y, p^\ast \LL_X)
\end{equation*}
for any $p : Y \to X$. This derived stack classifies the sections of $\LL_X \to \cO_X$ or, equivalently, the $1$-forms on $X$; see for example \cite[beginning of Section 2]{Calaque2019}. Note that when $p : \Spec A \to X$ is an affine point of $X$ then
\begin{align*}
	\dSt_{/X}(\Spec A, \bT^\ast X) & \simeq \QCoh(\Spec A)(\cO_{\Spec A}, p^\ast \LL_X) \\
	& \simeq \QCoh(\Spec A)(p^\ast \TT_X, \cO_{\Spec A}) \\
	& \simeq \Mod_A(p^\ast \TT_X, A),
\end{align*}
where $\TT_X := \LL_X^\vee \in \QCoh(X)$ is the dual of the cotangent stack. When $X = \Spec B$ then we set $\TT_B := \TT_{\Spec B} \in \Mod_B$ and we have $\TT_B \simeq \LL_B^\vee$ where $\LL_B$ is the corepresenting object for the functor of $\KK$-linear derivations, $\mathrm{Der}_\KK(B, -) : \Mod_B \to \Spaces$. Moreover
\begin{align*}
	\dSt_{/\Spec B}(\Spec A, \bT^\ast (\Spec B)) & \simeq \Mod_A(A \otimes_B \TT_B, A) \\
	& \simeq \Mod_B(\TT_B, p_\ast A) \\
	& \simeq \sCAlg_{\KK, B/}(\mathrm{Sym}(\TT_B), p_\ast A) \\
	& \simeq \dSt_{/\Spec B} (\Spec A, \Spec \mathrm{Sym}(\TT_B))
\end{align*}
and so
\begin{equation*}
	\bT^\ast (\Spec B) \simeq \Spec \mathrm{Sym}(\TT_B).
\end{equation*}
If we further restrict ourselves to $\KK^n := \Spec \KK[x_1, \dotsc, x_n]$ then a simple calculation shows that
\begin{equation*}
	\TT_{\KK[x_1, \dotsc, x_n]} \simeq \LL_{\KK[x_1, \dotsc, x_n]}^\vee \simeq \KK[x_1, \dotsc, x_n]^{\oplus n}
\end{equation*}
as $\KK[x_1, \dotsc, x_n]$-modules, so
\begin{align*}
	\bT^\ast \KK^n & \simeq \Spec \mathrm{Sym}(\TT_{\KK[x_1, \dotsc, x_n]}) \\
	& \simeq \Spec \mathrm{Sym}(\KK[x_1, \dotsc, x_n]^{\oplus n}) \\
	& \simeq \Spec \KK[x_1, \dotsc, x_n, p_{x_1}, \dotsc, p_{x_n}],
\end{align*}
which in particular is equivalent to $\KK^{2n}$. It is proven in \cite[Theorem 2.2]{Calaque2019} that for any Artin stack $X$ the cotangent stack $\bT^\ast X$ carries a symplectic form $\omega_X := d \lambda_X$, where $\lambda_X$ is a tautological $1$-form on $\bT^\ast X$ coming from the identity map $\bT^\ast X \to \bT^\ast X$ under the equivalence
\begin{equation*}
	\dSt_X(\bT^\ast X, \bT^\ast X) \simeq \QCoh(\bT^\ast X)(\cO_{\bT^\ast X}, \pi^\ast \LL_X) \simeq \QCoh(\bT^\ast X)(\cO_{\bT^\ast X}, \LL_{\bT^\ast X})
\end{equation*}
(here $\pi : \bT^\ast X \to X$ is the canonical projection). Thanks to the above calculation we can concentrate on the even-dimensional affine planes $\bT^\ast \KK^n \simeq \KK^{2n}$ whose canonical symplectic form is given (globally) by
\begin{equation*}
	\omega_n \simeq d \left( \sum_{i = 1}^n x_i \wedge dp_{x_i} \right) \simeq \sum_{i = 1}^n dx_i \wedge dp_{x_i}.
\end{equation*}
Therefore $(\KK^{2n}, \omega_n) \in \sCRW$ for every $n \geq 0$. The equivalences
\begin{equation} \label{eqn:product-of-affine-planes}
	\bT^\ast \KK^m \times \bT^\ast \KK^n \simeq \KK^{2m} \times \KK^{2n} \simeq \KK^{2(m+n)} \simeq \bT^\ast \KK^{m+n}
\end{equation}
send the product form $\omega_m + \omega_n$ to $\omega_{m+n}$, i.e. they are equivalences in of symplectic stacks. Moreover, let $(-)^\diamond : \PreSymp \to \PreSymp$ denote the dualization functor which switches the sign of the symplectic form (see \Cref{prp:duals}). Then we have an equivalence
\begin{equation} \label{eqn:cotangent-complex-self-dual}
	\bT^\ast \KK^n \simeq (\bT^\ast \KK^n)^\diamond
\end{equation}
in $\sCRW$ given by switching the sign of the ``momentum'' variables: $x_i \mapsto x_i$ and $p_{x_i} \mapsto - p_{x_i}$ for all $i = 1, \dotsc, n$.

\section{Matrix factorizations} \label{sec:matrix-fact}


\subsection{The $2$-category $\bMF$} \label{sec:bMF}

Throughout this section, and in fact the rest of the paper, we abbreviate tuples of variables by underlines and the length of tuples by norms: $\undl{x} := (x_1, \dotsc, x_n)$ and $\vnorm{x} := n$. Concatenation of tuples is denoted by juxtaposition: $\undl{xy} = (x_1, \dotsc, x_n, y_1, \dotsc, y_p)$.

Fix a polynomial $V \in \KK[\undl{x}]$. A \emph{matrix factorization of $V$} is a free $\ZZ_2$-graded $\KK[\undl{x}]$-module $M$ together with an odd endomorphism $d_M$ satisfying $d_M^2 = V \cdot \id_{M}$. The $\KK[\undl{x}]$-module $\Hom_{\KK[\undl{x}]}(M,N)$ can be given a differential $\delta$ defined by
\begin{equation*}
	\delta \vphi = d_N \vphi - (-1)^{\vnorm{\vphi}} \vphi d_M,
\end{equation*}
where $\vnorm{\vphi}$ is the degree of the map $\vphi$. In particular, $\End_{\KK[\undl{x}]}(M)$ has differential $\delta = [d_M, -]$, the graded commutator with $d_M$. Matrix factorizations form a dg-category $\cMF(\undl{x};V)$.

The equation $d^2 = V \cdot \id$ has some striking consequences. Enhance the category so that every matrix factorization is equipped with a choice of basis, turning every morphism $M \to N$ into a (possibly infinite) matrix. Then in particular $d_M : M \to M$ is a matrix with polynomial entries and so we can define
\begin{equation*}
	\lambda_{x_i} := \partial_{x_i} d_M	
\end{equation*}
by differentiating each entry of $d_M$ with respect to the variable $x_i$. If the variable $x_i$ is clear from the context then we use $\lambda_i$ instead of $\lambda_{x_i}$.

\begin{lmm}[{\cite[Lemma 2.11]{Murfet2018}}] \label{lmm:derivative-of-dm}
	The odd morphism $\lambda_{i} : M \to M$ is independent of the choice of basis of $M$ up to chain homotopy.
\end{lmm}

\begin{proof}
	Say $M_1$ and $M_2$ are two copies of $M$ equipped with two different bases, $d_1$ and $d_2$ are the matrices for $d_M$ with respect to those bases, and $P$ is a change-of-basis matrix for $d_M$: namely, $P d_1 = d_2 P$ holds. Set $\lambda_1 = \partial_{i} d_1$ and $\lambda_2 = \partial_{i} d_2$. Differentiating both sides of $P d_1 = d_2 P$ with respect to $x_i$ and using the product rule yields
	\begin{equation*}
		(\partial_{i} P) d_1 + P \lambda_1 = \lambda_2 P + d_2 (\partial_{i} P)
	\end{equation*}
	and so
	\begin{equation*}
		P \lambda_1 - \lambda_2 P = d_2 (\partial_{i} P) - (\partial_{i} P) d_1
	\end{equation*}
	as morphisms $M_1 \to M_2$. Thus any two matrix expressions for $\lambda_{i}$ are chain homotopic, as desired. 
\end{proof}

There is a symmetric monoidal $2$-category with duals, denoted $\bMF$, where matrix factorizations are the $2$-morphisms. This was constructed in \cite{BCR2023}. We recall its basic structure in preparation for the main theorem:

The objects of $\bMF$ are finite tuples of variables $\undl{x}$. The symmetric monoidal product is concatenation of tuples and the empty tuple $\emptyset$ acts as the unit. The $1$-morphisms between two tuples $\undl{x}$ and $\undl{y}$ are pairs $(\undl{a}, V)$, where $\undl{a}$ is a new collection of variables distinct from those in $\undl{xy}$ and $V \in \KK[\undl{xya}]$ is a polynomial. Horizontal composition is given by concatenating the extra variables, including the intermediary one, and summing the polynomials:
\begin{equation*}
	(\undl{b}, W(\undl{yzb})) \circ (\undl{a}, V(\undl{xya})) = (\undl{ayb}, V(\undl{xya}) + W(\undl{yzb})). 
\end{equation*}
If $(\undl{a}, V)$ and $(\undl{a'}, V')$ are $1$-morphisms $\undl{x} \to \undl{y}$, a $2$-morphism $(\undl{a}, V) \Rightarrow (\undl{a'}, V')$ is an isomorphism class of objects in the idempotent completion of $\bh \cMF(\undl{xyaa'}; V' - V)$, the homotopy category of the dg-category $\cMF(\undl{xyaa'}; V' - V)$ obtained by taking the even cohomology group of all the mapping complexes; more concretely, a $2$-morphism is a chain homotopy equivalence class $[M]$ of matrix factorizations of $V' - V$ such that one representative is a direct summand of a finite-rank matrix factorization. The vertical composition is given by the tensor product, which is compatible with homotopy equivalences: if $M : (\undl{a}, U) \Rightarrow (\undl{b}, V)$ and $N : (\undl{b}, V) \Rightarrow (\undl{c}, W)$ represent two $2$-morphisms then
\begin{equation*}
	N \circ M = M \otimes_{\KK[\undl{xyb}]} N
\end{equation*}
is a representative for the composite $2$-morphism, with the twisted differential on the right given by $d_M \otimes 1 + 1 \otimes d_N$.

The unit for the horizontal composition of $1$-morphisms is
\begin{equation*}
	\id_{\undl{x}} := (\undl{a}, \sum_{i = 1}^{\vnorm{x}} a_i(\wilde{x}_i - x_i));
\end{equation*}
the idea is that when composing a $1$-morphism with the unit we're substituting the variables $x_i$ with the variables $\wilde{x}_i$. Note that this is not a strict unit -- see for example the end of Section 2.1 in \cite{BCR2023} for an explicit unitor isomorphism. The unit for the vertical composition of $2$-morphisms is (the isomorphism class of) $I_{(\undl{a}, V)}$, the Koszul matrix factorization described in Section 2.1 of \cite{BCR2023} and later here in the following subsection: indeed, for any $2$-morphism $[M] : (\undl{a}, V) \Rightarrow (\undl{a'}, V')$ we have
\begin{equation*}
	[M] \circ [I_{(\undl{a}, V)}] = [I_{(\undl{a}, V)} \otimes_{\KK[\undl{xya}]} M] = [M] = [M \otimes_{\KK[\undl{xya'}]} I_{(\undl{a'}, V')}] = [I_{(\undl{a'}, V')}] \circ [M].
\end{equation*}
Explicit formulas for the homotopy equivalences in this equation can be found in \cite{CM2016,CM2020}.

\subsection{A useful quasi-isomorphism}

Fix a $1$-morphism $(\undl{a}, V) : \undl{x} \to \undl{y}$ in $\bMF$, with $\vnorm{\undl{a}} = k$. We need an explicit model for $I_{(\undl{a}, V)}$ as a matrix factorization of $V(\undl{xya'}) - V(\undl{xya})$ -- notice the extra tuple of variables $\undl{a'}$, which has the same length as $\undl{a}$. As a free $\KK[\undl{xyaa'}]$-module it is given by
\begin{equation*}
	I_{(\undl{a}, V)} = \KK[\undl{xyaa'}; \undl{\theta}] := \KK[\undl{xyaa'}] \otimes_\KK \Lambda (\theta_1, \dotsc, \theta_k),
\end{equation*}
where the right factor in the tensor product denotes the exterior algebra of a $\KK$-vector space with basis elements $\theta_1, \dotsc, \theta_k$. The exterior algebra inherits a $\ZZ_2$-grading where the elements $\theta_i$ are naturally in odd degree, so $I_{(\undl{a},V)}$ is a $\ZZ_2$-graded $\KK[\undl{xyaa'}]$-module. The differential is
\begin{equation*}
	d_{I_{(\undl{a}, V)}} = \sum_{i = 1}^k p_{i,V} \theta_i + (a_i' - a_i) \frac{\partial}{\partial \theta_i},
\end{equation*}
where
\begin{equation*}
	p_{i,V} := \frac{V(\undl{xy},a_1, \dotsc, a_{i-1}, a_i', \dotsc, a_k') - V(\undl{xy},a_1, \dotsc, a_{i}, a_{i+1}', \dotsc, a_k')}{a_i' - a_i} \in \KK[\undl{xyaa'}]
\end{equation*}
is the difference quotient whose limit as $a_i' \to a_i$ results in $\partial_{a_i} V$. It's easy to see that
\begin{equation*}
	d_{I_{(\undl{a},V)}}^2 = \sum_{i = 1}^k p_{i,V} \cdot (a_i' - a_i) = V(\undl{xya'}) - V(\undl{xya})
\end{equation*}
and so $I_{(\undl{a}, V)}$ is a matrix factorization of the correct polynomial.

Now we can state an important lemma. If $M$ is a matrix factorization, we denote by $\End(M)$ the $\ZZ_2$-graded cochain complex of endomorphisms of $M$.
\begin{lmm} \label{lmm:quasi-iso-end}
	There is an equivalences of algebras
	\begin{equation*}
		\End(I_{(\undl{a}, V)}) \simeq R_{(\undl{a}, V)}
	\end{equation*}
	where $R_{(\undl{a}, V)}$ denotes a cofibrant resolution of $\KK[\undl{xya}]/\langle \partial_{a_1} V, \dotsc, \partial_{a_k} V \rangle$ (a cdga concentrated in degree $0$) in the model category of $\ZZ_2$-graded commutative dgas.
\end{lmm}

In particular, since $R_{(\undl{a}, V)}$ is commutative, this result shows that the natural multiplication of $\End(I_{(\undl{a}, V)})$ given by composition of endomorphisms is homotopy commutative. This is to be expected from an Eckmann-Hilton argument: since $I_{(\undl{a}, V)}$ is an identity morphism in a $2$-category its endomorphisms inherit a second multiplication which is compatible with the composition, and Eckmann-Hilton precisely ensures that the two multiplications coincide and are both commutative.

First, here is an explicit model for $R_{(\undl{a}, V)}$:
\begin{equation*}
	R_{(\undl{a}, V)} = \KK[\undl{xya}; \undl{\alpha}]
\end{equation*}
where $\vnorm{\alpha} = \vnorm{a}$ and with differential $d_{R_{(\undl{a};V)}}$ given on generators by
\begin{equation*}
	d_{R_{(\undl{a};V)}} p(\undl{xya}) = 0, \quad d_{R_{(\undl{a};V)}}\alpha_i = \partial_{a_i} V,
\end{equation*}
for any $p(\undl{xya}) \in \KK[\undl{xya}]$ and any $1 \leq i \leq k$, and then extended using the Leibniz rule.

\begin{proof}[Proof of \Cref{lmm:quasi-iso-end}]
	We will build a zigzag of equivalences
	\begin{equation} \label{eqn:zigzag}
		 \End(I_{(\undl{a}, V)}) \xlra{\simeq} \End(I_{(\undl{a}, V)})[\undl{\beta}] \xlla{\simeq} R_{(\undl{a}, V)},
	\end{equation}
	which will prove the existence of the desired equivalence $\End(I_{(\undl{a}, V)}) \simeq R_{(\undl{a}, V)}$. In the middle we have
	\begin{equation*}
		\End(I_{(\undl{a}, V)})[\undl{\beta}] := \End(I_{(\undl{a}, V)}) \otimes_{\KK} \Lambda(\beta_1, \dotsc, \beta_k),
	\end{equation*}
	where $\vnorm{\undl{a}} = k$, and we define the differential $d$ on it to be the usual differential $\delta = [d_{I_{(\undl{a},V)}}, -]$ (the commutator with Koszul signs) on $\End(I_{(\undl{a}, V)})$ and $d\beta_i = \partial_{a_i} V$ on each generator.
	
	There is an inclusion of dgas
	\begin{equation*}
		\End(I_{(\undl{a}, V)}) \to \End(I_{(\undl{a}, V)})[\undl{\beta}], \quad f \mapsto f \otimes 1
	\end{equation*}
	which is also an equivalence, i.e. an isomorphism in cohomology: the only new differentials are those coming from each $\beta_i$, but since
	\begin{equation*}
		d \beta_i = \partial_{a_i} V = [d_{I_{(\undl{a},V)}}, \partial_{a_i} d_{I_{(\undl{a},V)}}] = \delta( \partial_{a_i} d_{I_{(\undl{a},V)}} )
	\end{equation*}
	we find that they were already accounted for in the cohomology of $\End(I_{(\undl{a}, V)})$. The second equality can be obtained by differentiating the equation $(d_{I_{(\undl{a},V)}})^2 = V$ by the variable $a_i$. This yields the first equivalence in \Cref{eqn:zigzag}.
	
	The second map $t : R_{(\undl{a};V)} \to \End(I_{(\undl{a},V)})[\undl{\beta}]$ is given as follows: using the model $R_{(a, V)} = \KK[\undl{xya}; \undl{\alpha}]$ described above, $t$ acts on generators by
	\begin{equation*}
		x_i \mapsto x_i, \quad y_i \mapsto y_i, \quad a_i \mapsto \frac{a_i + a_i'}{2}, \quad \alpha_i \mapsto \beta_i.
	\end{equation*}
	The rest of the map is determined by extending linearly over $\KK$ and using the graded Leibniz rule. We compute the cohomology on both sides to show that this map is an equivalence. Since $R_{(\undl{a};V)}$ is a resolution of $\KK[\undl{xya}]/\langle \partial_{a_1} V, \dotsc, \partial_{a_k} V \rangle$ we have
	\begin{equation*}
		H^0(R_{(\undl{a};V)}) \cong \KK[\undl{xya}]/\langle \partial_{a_1} V, \dotsc, \partial_{a_k} V \rangle, \quad H^1(R_{(\undl{a};V)}) \cong 0.
	\end{equation*}
	The cohomology of $\End(M)$, for general $M$, can be found in \cite[Section 4]{KR2008}, and it agrees with the cohomology of $R_{(\undl{a};V)}$ when $M = I_{(\undl{a}, V)}$. In fact, the variables $a_i'$ disappear in cohomology because $d\theta_i = a_i' - a_i$, and the partial derivatives $\partial_{a_i} V$ get quotiented out because $d (\partial/\partial \theta_i) = p_{i, V} = \partial_{a_i} V$ after identifying $a_i'$ with $a_i$. After these identifications in cohomology the map $H^0(t)$ becomes the identity map on $\KK[\undl{xya}]$ and correctly kills each partial derivative $\partial_{a_i} V$ by sending $\alpha_i$ to $\beta_i$, which shows that it is an isomorphism. We are now done since the odd cohomologies of $R_{(\undl{a};V)}$ and $\End(I_{(\undl{a},V)})[\undl{\beta}]$ are both zero.
\end{proof}

\section{Constructing the functor} \label{sec:functor}

This last section is devoted to the construction of the functor $\fe : \bMF \to \bh_2 \sCRW$. We will do this in pieces by prescribing, for each $n$, the value of $\fe$ on $n$-cells ($n = 0,1,2$) and supplying the structure morphisms necessary for a symmetric monoidal $2$-functor.

\subsection{From variables to derived symplectic stacks} \label{sec:e-0}

For an object $\undl{x} \in \bMF$ of length $m$ we define $\fe(\undl{x}) := \bT^\ast \KK^m \in \sCRW$, which we know from \Cref{sec:affine-planes} to be equivalent to $\Spec \KK[\undl{x}, \undl{p_x}]$. There are canonical equivalences $\bT^\ast \KK^{m+n} \simeq \bT^\ast \KK^m \times \bT^\ast \KK^n$ of derived symplectic stacks (\Cref{eqn:product-of-affine-planes}), so $\fe$ canonically preserves the monoidal product. It also preserves the unit: $\fe(\emptyset) \simeq \bT^\ast \KK^0 \simeq \ast$.

\subsection{From polynomials to Lagrangian morphisms}

For a polynomial $(\undl{a}, V) : \undl{x} \to \undl{y}$ we need $\fe(\undl{a}, V)$ to be a Lagrangian span
\begin{equation*}
	\begin{tikzcd}[column sep = small, row sep = small]
		& L_{(\undl{a}, V)} \ar[dl] \ar[dr] & \\
		\bT^\ast \KK^m & & \bT^\ast \KK^n
	\end{tikzcd}
\end{equation*}
or alternatively, using \Cref{eqn:product-of-affine-planes,eqn:cotangent-complex-self-dual}, a Lagrangian morphism
\begin{equation*}
	L_{(\undl{a}, V)} \to \bT^\ast \KK^{m+n}.
\end{equation*}
Set
\begin{equation*}
	L_{(\undl{a},V)} := \Spec R_{(\undl{a},V)}
\end{equation*}
where $R_{(\undl{a},V)}$ is the cdga introduced in \Cref{lmm:quasi-iso-end} whose underlying $\KK[\undl{xya}]$-module is $\KK[\undl{xya};\undl{\alpha}]$ and whose differential is induced by $d \alpha_i = \partial_{a_i} V$.

\begin{prp} \label{prp:e-1cells}
	Define the cdga maps
	\begin{align*}
		f_{\undl{x}} : \KK[\undl{x}, \undl{p_x}] & \to R_{(\undl{a},V)}, & f_{\undl{x}}(x_i) = x_i, \quad f_{\undl{x}}(p_{x_i}) = -\partial_{x_i} V, \\
		f_{\undl{y}} : \KK[\undl{y}, \undl{p_y}] & \to R_{(\undl{a},V)}, & f_{\undl{y}}(y_i) = y_i, \quad f_{\undl{y}}(p_{y_i}) = \partial_{y_i} V.
	\end{align*}
	The induced morphism
	\begin{equation*}
		f : L_{(\undl{a},V)} \to \Spec \KK[\undl{xy}, \undl{p_x p_y}] \simeq \bT^\ast \KK^{m+n}
	\end{equation*}
	carries the structure of a Lagrangian morphism of derived stacks.
\end{prp}

Before starting the proof we recall a useful construction. A polynomial $p \in \KK[\undl{u}]$ induces a map
\begin{equation*}
	\gamma_p : \KK[\undl{u}, \undl{p_u}] \to \KK[\undl{u}], \qquad \gamma_p(u_i) = u_i, \quad \gamma_p(p_{u_i}) = \partial_{u_i} p
\end{equation*}
which passes to a map of derived stacks
\begin{equation*}
	\Gamma_p : \KK^{\vnorm{\undl{u}}} \simeq \Spec \KK[\undl{u}] \to \Spec \KK[\undl{u}, \undl{p_u}] \simeq \bT^\ast \KK^{\vnorm{\undl{u}}}
\end{equation*}
called the \emph{graph of $dp$}; here $dp$ is a $1$-form on $\KK^{\vnorm{\undl{u}}}$ canonically obtained from $p$ (see \cite{Calaque2015}). By \cite[Theorem 2.15]{Calaque2019} there is a canonical Lagrangian structure on $\Gamma_p$ induced by the nullhomotopy $d(dp) \simeq 0$.

\begin{proof}[Proof of \Cref{prp:e-1cells}]
	We will write $f$ as the pullback of two Lagrangian spans, which is again Lagrangian by \cite[Theorem 4.4]{Calaque2015}. Consider the two polynomials $0 \in \KK[\undl{a}]$ and $V \in \KK[\undl{xya}]$. They induce Lagrangian morphisms
	\begin{equation*}
		\Gamma_0 : \KK^{k} \to \bT^\ast \KK^k, \quad \Gamma_V : \KK^{m+n+k} \to \bT^\ast \KK^{m+n+k} \simeq (\bT^\ast \KK^k)^\diamond \times \bT^\ast \KK^{m+n},
	\end{equation*}
	namely the graphs of $d0$ and $dV$, which we can convert to two Lagrangian spans
	\begin{equation*}
		\begin{tikzcd}
			& \KK^k \ar[dl] \ar[dr] & & \KK^{m+n+k} \ar[dr] \ar[dl] & \\
			\ast & & \bT^\ast \KK^k & & \bT^\ast \KK^{m+n}
		\end{tikzcd}
	\end{equation*}
	after flipping some signs (see \Cref{eqn:cotangent-complex-self-dual}). The derived pullback of these spans is
	\begin{equation*}
		\begin{tikzcd}
			& \Spec (\KK[\undl{a}] \otimes_{\KK[\undl{a}, \undl{p_a}]} \KK[\undl{xya}]) \ar[dl] \ar[dr, "g"] & \\
			\ast & & \bT^\ast \KK^{m+n}
		\end{tikzcd}
	\end{equation*}
	and it can be easily verified that
	\begin{equation*}
		R_{(\undl{a}, V)} \simeq \KK[\undl{a}] \otimes_{\KK[\undl{a}, \undl{p_a}]} \KK[\undl{xya}].
	\end{equation*}
	One way to see this, for example, is to replace the map $\KK[\undl{a}, \undl{p_a}] \to \KK[\undl{a}]$ by the quasi-isomorphic cofibration $\KK[\undl{a}, \undl{p_a}] \to \KK[\undl{a},\undl{p_a};\undl{\alpha}]$, where the latter cdga has differential $d\alpha_i = p_{a_i} - \partial_{a_i} V$, and then computing the ordinary tensor product of cdgas. The result is isomorphic to the presentation of $R_{(\undl{a},V)}$ given in the previous section. Moreover, $g \simeq f$ after the identification, and so we're done.
\end{proof}

Geometrically we can think of $L_{(\undl{a}, V)}$ as a (Lagrangian) family of Lagrangian graphs parametrized by the parameters $\undl{a} \in \KK^k$. Indeed, instantiating the variables $a_i$ at points $A_i \in \KK$ will return an honest Lagrangian submanifold of $\bT^\ast \KK^{m+n}$, namely the graph of the differential of the polynomial $V(\undl{xy}, A_1, \dotsc, A_k) \in \KK[\undl{xy}]$.

Now we can define $\fe(\undl{a},V)$ to be the Lagrangian span
\begin{equation*}
	\begin{tikzcd}
		& L_{(\undl{a}, V)} \ar[dl] \ar[dr] & \\
		\bT^\ast \KK^m & & \bT^\ast \KK^n
	\end{tikzcd}
\end{equation*}
described above. The following properties of $L_{(\undl{a},V)}$ are just a matter of computing some derived tensor products:
\begin{prp} \label{prp:functoriality-1cells}
	Let $(\undl{a}, V) : \undl{x} \to \undl{y}$, $(\undl{a'}, V') : \undl{x'} \to \undl{y'}$, and $(\undl{b}, W) : \undl{y} \to \undl{z}$ be $1$-morphisms in $\bMF$, and let $o = \vnorm{\undl{z}}$, $m' = \vnorm{\undl{x'}}$, $n' = \vnorm{\undl{y'}}$. 
	\begin{enumerate}
		\item $\fe$ preserves the horizontal composition and the monoidal product of $1$-morphisms, i.e. we have equivalences
		\begin{align*}
			\fe(\undl{a}, V) \times_{\bT^\ast \KK^{n}} \fe(\undl{b}, W) & \simeq \fe(\undl{ab}, V + W) & \text{as spans from $\bT^\ast \KK^m$ to $\bT^\ast \KK^o$}, \\
			\fe(\undl{a}, V) \times \fe(\undl{a'}, V') & \simeq \fe(\undl{aa'}, V + V') & \text{as spans from $\bT^\ast \KK^{m+m'}$ to $\bT^\ast \KK^{n+n'}$}.
		\end{align*}
		\item $\fe$ preserves the units for the above operations, i.e. we have equivalences
		\begin{align*}
			\fe(\id_{\undl{x}}) & \simeq \text{diagonal span from $\bT^\ast \KK^m$ to itself}, \\
			\fe(\emptyset, 0) & \simeq \text{diagonal span from $\ast$ to itself}.
		\end{align*}
	\end{enumerate}
\end{prp}
These equivalences are essentially unique, since the two sides are presentations of the limit of the same diagram and two limits can be identified with each other via an essentially unique equivalence.

\subsection{From matrix factorizations to dg-modules}

The next step is to define $\fe$ on $2$-morphisms, so matrix factorizations. Recall from \Cref{eqn:cat-of-2-morphisms} that the $\infty$-category of $2$-morphisms between $\fe(\undl{a}, V) : \bT^\ast \KK^m \to \bT^\ast \KK^n$ and $\fe(\undl{b}, W) : \bT^\ast \KK^m \to \bT^\ast \KK^n$ is the $\infty$-category $\Mod^{\ZZ_2}_{A_{W - V}}$ of $\ZZ_2$-graded dg-modules over the cdga 
\begin{equation*}
	A_{W-V} := R_{(\undl{a}, V)} \otimes_{\KK[\undl{xy}, \undl{p_x p_y}]} R_{(\undl{b}, W)}.
\end{equation*}
This algebra has a nice explicit model: 
\begin{equation*}
	\KK[\undl{xyab}; \undl{\chi \upsilon \alpha \beta}]
\end{equation*}
as a $\ZZ_2$-graded exterior algebra over $\KK[\undl{xyab}]$ with one Greek variable $\chi_i$, $\upsilon_i$, $\alpha_i$, and $\beta_i$ in odd degree for each corresponding Roman variable $x_i$, $y_i$, $a_i$, and $b_i$ in even degree. The differential is given on generators by
\begin{equation*}
	d\chi_i = \partial_{x_i} W - \partial_{x_i} V, \quad d \upsilon_i = \partial_{y_i} W - \partial_{y_i} V, \quad d \alpha_i = - \partial_{a_i} V, \quad d\beta_i = \partial_{b_i} W.
\end{equation*}
More succinctly, if $\tau$ is a Greek variable corresponding to a Roman variable $t$ then $d\tau = \partial_t (W - V)$. Therefore $A_{W - V}$ is the algebra of functions on the full derived critical locus of the polynomial $W - V \in \KK[\undl{xyab}]$. Notice that we omit the extra variables $\undl{a}$ and $\undl{b}$ from the notation $A_{W - V}$ unlike with $R_{(\undl{a}, V)}$ and $R_{(\undl{b}, W)}$; this is because we don't have to keep track of which variables to differentiate with respect to, as we are including all of them.

\begin{prp} \label{prp:end-is-module}
	Let $M$ be a representative for a $2$-morphism $(\undl{a}, V) \to (\undl{b}, W)$, namely a matrix factorization of $W - V \in \KK[\undl{xyab}]$. Then the $\ZZ_2$-graded cochain complex $\End(M)$ inherits a canonical $A_{W - V}$-action, i.e. $\End(M) \in \Mod_{A_{W - V}}^{\ZZ_2}$.
\end{prp}

\begin{proof}
	We have a quasi-isomorphism
	\begin{equation*}
		\End(M) \simeq \End(I_{(\undl{a}, V)} \otimes_{\KK[\undl{xya}]} M \otimes_{\KK[\undl{xyb}]} I_{(\undl{b},W)})
	\end{equation*}
	since $I_{(\undl{a}, V)}$ and $I_{(\undl{b},W)}$ are units for the composition of $2$-morphisms in $\bMF$. The right-hand side has a canonical action of $\End(I_{(\undl{a}, V)})$ from the left factor and of $\End(I_{(\undl{b},W)})$ from the right factor in the tensor product. The two actions agree on $\KK[\undl{xy}]$ and can be further made to agree on $\KK[\undl{xy}, \undl{p_x p_y}]$: the action of $p_{x_i}$ is multiplication by $\partial_{x_i}(W - V)$, which is nulhomotopic as an endomorphism of $M$ (as can be seen by differentiating the equation $d_M^2 = W - V$), and similarly for $p_{y_i}$. Thus $\End(M)$ inherits an action of
	\begin{equation*}
		\End(I_{(\undl{a}, V)}) \otimes_{\KK[\undl{xy},\undl{p_x p_y}]} \End(I_{(\undl{b},W)}) \simeq R_{(\undl{a}, V)} \otimes_{\KK[\undl{xy},\undl{p_x p_y}]} R_{(\undl{b}, W)} = A_{W - V}
	\end{equation*}
	where the first equivalence is due to \Cref{lmm:quasi-iso-end}.
\end{proof}

One might reasonably object to the usage of quasi-isomorphisms without unique inverses to define this action. This is justified by the fact that we are working in an $\infty$-category, so a module over an algebra needs to be interpreted in the correct homotopical sense, with the structure diagrams for the algebra action commuting up to coherent homotopies (in this case, coherent quasi-isomorphisms). Since $\End(-)$ and $- \otimes -$ are sufficiently functorial and in particular respect quasi-isomorphisms, this is not a problem for us: in the homotopy $2$-category $\bh_2 \sCRW$ the homotopies turn into coherent isomorphisms, and so are appropriately invertible. 

Using an explicit homotopy equivalence $M \simeq I_{(\undl{a}, V)} \otimes_{\KK[\undl{xya}]} M \otimes_{\KK[\undl{xyb}]} I_{(\undl{b},W)}$ it would be possible to rigidify the situation to obtain explicit formulas for the action of $A_{W - V}$ on $\End(M)$. Despite our efforts we couldn't find formulas that were satisfyingly clear or elucidating, so we are happy to leave this task to the interested reader. We expect, however, that they will reflect the formulas for the Clifford action described in \cite{Murfet2018}, since the cut operation presented in that paper is one way to ``wedge'' the unit $I_{(\undl{a}, V)}$ inside the composition of two matrix factorizations.

Finally, we set $\fe(M) := \End(M)$ with the $A_{W - V}$-action as in the proposition. It is clear that this assignment is independent of the chosen representative within the class $[M]$, so we omit the square brackets for simplicity. We have the following properties.
\begin{prp} \label{prp:functoriality-2cells}
	Fix objects $\undl{x}, \undl{y}, \undl{x'}, \undl{y'}, \undl{z}$ and $1$-morphisms $(\undl{a}, V), (\undl{b}, W), (\undl{c}, X) : \undl{x} \to \undl{y}$ and $(\undl{a'}, V'), (\undl{b'}, W') : \undl{y} \Rightarrow \undl{z}$ and $(\undl{a''}, V''), (\undl{b''}, W'') : \undl{x'} \Rightarrow \undl{y'}$ Let $M : (\undl{a}, V) \Rightarrow (\undl{b}, W)$, $N : (\undl{b}, V) \Rightarrow (\undl{c}, X)$, $M' : (\undl{a'}, V') \Rightarrow (\undl{b'}, W')$, and $M'' : (\undl{a''}, V'') \Rightarrow (\undl{b''}, W'')$ be representatives for $2$-morphisms in $\bMF$.
	\begin{enumerate}
		\item $\fe$ preserves the horizontal composition, the vertical composition, and the product of $2$-morphisms, i.e. we have equivalences
		\begin{align*}
			\fe(M) \otimes_{\KK[\undl{y}, \undl{p_y}]} \fe(M') & \simeq \fe(M \otimes_{\KK[\undl{y}]} M') & \text{as $A_{W + W' - (V + V')}$-modules}, \\
			\fe(M) \otimes_{R_{(\undl{b}, W)}} \fe(N) & \simeq \fe(M \otimes_{\KK[\undl{xyb}]} N) & \text{as $A_{X - V}$-modules}, \\
			\fe(M) \otimes_\KK \fe(M'') & \simeq \fe(M \otimes_{\KK} M'') & \text{as $A_{W + W'' - (V + V'')}$-modules}.
		\end{align*}
		\item $\fe$ preserves the units for the above operations, i.e. we have equivalences
		\begin{align*}
			\fe(I_{\id_{\undl{x}}}) & \simeq \KK[\undl{x}, \undl{p_x}] & \text{as $\KK[\undl{x}, \undl{p_x}] \otimes_\KK \KK[\undl{x}, \undl{p_x}]$-modules}, \\
			\fe(I_{(\undl{a}, V)}) & \simeq R_{(\undl{a}, V)} & \text{as $R_{(\undl{a}, V)} \otimes_\KK R_{(\undl{a}, V)}$-modules}, \\
			\fe(I_{(\emptyset, 0)}) & \simeq \KK & \text{as $\KK$-modules}.
		\end{align*}
	\end{enumerate}
\end{prp}

\begin{proof}
To obtain the module structures as in the right column, note that an $A_{W - V}$-module and a $A_{X - W}$-module can be tensored to obtain an $A_{X - W} \otimes_{R_{(\undl{b}, W)}} A_{W - V}$-module and the action can be further restricted to that of a $A_{X - V}$-module via the canonical composite map
\begin{align*}
	A_{X - V} & \simeq R_{(\undl{a}, V)} \otimes_{\KK[\undl{xz}, \undl{p_x p_z}]} R_{(\undl{c}, X)} \\
	& \to R_{(\undl{a}, V)} \otimes_{\KK[\undl{xy}, \undl{p_x p_y}]} R_{(\undl{b}, W)} \otimes_{\KK[\undl{yz}, \undl{p_y p_z}]} R_{(\undl{c}, X)} \\
	& \simeq (R_{(\undl{a}, V)} \otimes_{\KK[\undl{xy}, \undl{p_x p_y}]} R_{(\undl{b}, W)}) \otimes_{R_{(\undl{b}, W)}} (R_{(\undl{b}, W)} \otimes_{\KK[\undl{yz}, \undl{p_y p_z}]} R_{(\undl{c}, X)}) \\
	& \simeq A_{W - V} \otimes_{R_{(\undl{b}, W)}} A_{X - W}.
\end{align*}
This corresponds to the composition of $2$-morphisms in $\sCRW$ that was reviewed in \Cref{sec:affine-part}.

The equivalences in (2) can be verified by computation: the second one is \Cref{lmm:quasi-iso-end} and the other two follow from the second one. The equivalences in (1) require the quasi-isomorphism
\begin{equation*}
	\End(A) \simeq A \otimes_{\KK[\undl{w}]} A^\vee
\end{equation*}
which exists for cohomologically finite matrix factorizations $A$ of a polynomial $p \in \KK[\undl{w}]$. Here $M^\vee$ is (possibly a shift of) the linear dual of $M$ as a matrix factorization and the equivalence is the usual (signed) evaluation map -- see \cite[Section 4]{KR2008} for example. In particular this holds for all the $2$-morphisms in $\bMF$ by definition, see the beginning of \Cref{sec:bMF}. We can now write down the equivalences in (1) as chains of equivalences involving the matrix factorizations and their duals. The technique is the same for all three, so we only show the second one:
\begin{align*}
	\fe(M) \otimes_{R_{(\undl{b}, W)}} \fe(N) & = \End(M) \otimes_{R_{(\undl{b}, W)}} \End(N) \\
	& \simeq (M \otimes_{\KK[\undl{xyab}]} M^\vee) \otimes_{R_{(\undl{b}, W)}} (N \otimes_{\KK[\undl{xybc}]} N^\vee) \\
	& \simeq (M \otimes_{\KK[\undl{xyb}]} N) \otimes_{\KK[\undl{xyabc}]} (M \otimes_{\KK[\undl{xyb}]} N)^\vee \\
	& \simeq \End(M \otimes_{\KK[\undl{xyb}]} N) \\
	& = \fe(M \otimes_{\KK[\undl{xyb}]} N).
\end{align*}
It's straightforward to see that the actions of $R_{(\undl{a}, V)}$ and $R_{(\undl{c}, W)}$ commute with all equivalences (the first and last two are tautological, the third is a reordering), so in the end we obtain an equivalence of modules over $R_{(\undl{a}, V)} \otimes_{\KK[\undl{xz}, \undl{p_x p_z}]} R_{(\undl{c}, X)} \simeq A_{X - V}$, as needed.
\end{proof}

\subsection{The functor}

The results of the previous three subsections imply the following result:
\begin{thm} \label{thm:main}
	There is a symmetric monoidal functor
	\begin{equation*}
		\fe : \bMF \to \bh_2 \sCRW
	\end{equation*}
	from the $2$-category of matrix factorizations to the homotopy $2$-category of $\sCRW$ whose values on $k$-morphisms are given by
	\begin{align*}
		\fe(\undl{x}) & = \bT^\ast \KK^{\vnorm{\undl{x}}} & \text{on objects}, \\
		\fe(\undl{a}, V) & = \left(
		{\begin{tikzcd}[ampersand replacement = \&, column sep=tiny, row sep = small]
			\& \Spec R_{(\undl{a}, V)} \ar[dr] \ar[dl] \& \\
			\bT^\ast \KK^{\vnorm{\undl{x}}} \& \& \bT^\ast \KK^{\vnorm{\undl{y}}}
		\end{tikzcd}}
		\right)
		& \text{on $1$-morphisms $\undl{x} \to \undl{y}$}, \\
		\fe(M) & = \End(M) \text{ (as a $A_{W - V}$-module)} & \text{on $2$-morphisms $(\undl{a}, V) \Rightarrow (\undl{b}, W)$}.
	\end{align*}
\end{thm}

\begin{proof}
	The equivalences described in \Cref{sec:e-0}, \Cref{prp:functoriality-1cells}, and \Cref{prp:functoriality-2cells} are all the structure isomorphisms required for the assignment $\fe$ in the claim to be a functor of symmetric monoidal $2$-categories. Since these equivalences reduce to the structure equivalences for the pullback of stacks and the tensor product of modules, which are appropriately unital and associative, they satisfy all the necessary unit, associativity, and coherence identities in $\bh_2 \sCRW$ as well.
\end{proof}

\end{document}